\documentclass[12pt]{amsart}

\usepackage[top=2.54cm, bottom=2.54cm, left=2.6cm, right=2.6cm]{geometry}
\usepackage{lscape} 
\usepackage{bm} 
\usepackage{enumerate}
\usepackage{amssymb,amsmath,amsthm}		
\usepackage{tikz-cd}
\usetikzlibrary{matrix,arrows,decorations.pathmorphing}
\usepackage{mathrsfs} 
\numberwithin{equation}{section} 
\usepackage{graphicx}
\usepackage[titletoc,title]{appendix}

\newtheorem{thm}{Theorem}[section]
\newtheorem{prop}[thm]{Proposition}
\newtheorem{lem}[thm]{Lemma}
\theoremstyle{definition}
\newtheorem{defi}[thm]{Definition}

\theoremstyle{remark}
\newtheorem{rmk}[thm]{Remark}
\theoremstyle{plain}

\DeclareMathOperator{\Mat}{Mat}
\DeclareMathOperator{\Ext}{Ext}
\DeclareMathOperator{\tr}{tr}
\DeclareMathOperator{\tor}{tor}

\title[A $t$-motivic interpretation of shuffle relations]{A $t$-motivic interpretation of shuffle relations for multizeta values}
\author{Wei-Cheng Huang}
\date{}
\thanks{The author was partially supported by Prof. C.-Y. Chang's MOST Grant 102-2115-M-007-013-MY5}

\address{Department of Mathematics, Texas A\&M University, College Station, TX 77843, U.S.A.}
\email{wchuang@math.tamu.edu}

\begin{document}
\begin{abstract}
Thakur \cite{Thakurshuffle} showed that, for $r,$ $s\in \mathbb{N}$, a product of two Carlitz zeta values $\zeta_A(r)$ and $\zeta_A(s)$ can be expressed as an $\mathbb{F}_p$-linear combination of $\zeta_A(r+s)$ and double zeta values of weight $r+s$. Such an expression is called shuffle relation by Thakur. Fixing $r,$ $s\in \mathbb{N}$, we construct a $t$-module $E'$. To determine whether an $(r+s)$-tuple $\mathfrak{C}$ in $\mathbb{F}_q(\theta)^{r+s}$ gives a shuffle relation, we relate it to the $\mathbb{F}_q[t]$-torsion property of the point $\mathbf{v}_\mathfrak{C}\in E'(\mathbb{F}_q[\theta])$ constructed with respect to the given $(r+s)$-tuple $\mathfrak{C}$. We also provide an effective criterion for deciding the $\mathbb{F}_q[t]$-torsion property of the point $\mathbf{v}_\mathfrak{C}$.
\end{abstract}

\maketitle

\section{Introduction}
Let $A:=\mathbb{F}_q[\theta]$ be the polynomial ring in the variable $\theta$ over the finite field of $q$ elements $\mathbb{F}_q$ where $q$ is a power of a prime $p$. We denote by $A_+$ the set of monic polynomials in $A$ and let $k:=\mathbb{F}_q(\theta)$ be the field of fractions of $A$. For $(s_1,\dots,s_r)\in \mathbb{N}^r$, Thakur \cite{Thakurfunctionfield} introduced the multizeta value $\zeta_A(s_1,\dots,s_r)$ defined by \[\zeta_A(s_1,\dots,s_r):=\sum_{\substack{(a_1,\dots,a_r)\in A^r_+\\ \deg a_1>\cdots>\deg a_r}}\frac{1}{a_1^{s_1}\cdots a_r^{s_r}}\in \mathbb{F}_q(\!( 1/\theta)\!).\] Here $\sum_{i=1}^r s_i$ is called the \textit{weight} and $r$ is called the \textit{depth} of the presentation $\zeta_A(s_1,\dots,s_r).$ In particular, depth one multizeta values are called Carlitz zeta values initiated by Carlitz \cite{Car35} and depth two multizeta values are called double zeta values. In  \cite{Thakur09}, Thakur showed that each multizeta value is non-vanishing.

We fix positive integers $r,s\in\mathbb{N}$ and let $n:=r+s.$ By \cite{Thakurshuffle}, we know that the product of two zeta values $\zeta_A(r)$ and $\zeta_A(s)$ can be expressed as an $\mathbb{F}_p$-linear combination of $\zeta_A(n)$ and double zeta values of weight $n$. Such an expression is called shuffle relation by Thakur. Chen \cite{Chen2015153} derived an explicit formula of a shuffle relation, which is given by
\begin{align}\begin{split}\label{chen}
\zeta_A(r)\zeta_A(s)-&\zeta_A(r,s)-\zeta_A(s,r)=\\
&\zeta_A(n)+\sum_{\substack{i+j=n\\ (q-1)\mid j}}\left[ (-1)^{s-1}\begin{pmatrix}j-1\\s-1\end{pmatrix}+(-1)^{r-1}\begin{pmatrix}j-1\\r-1\end{pmatrix}\right]\zeta_A(i,j).
\end{split}\end{align} We want to study shuffle relations with coefficients in $k$. By a shuffle relation for multizeta values over $k$, we mean that the following identity holds: 
\begin{equation}\label{shueq}
\zeta_A(r)\zeta_A(s)-\zeta_A(r,s)-\zeta_A(s,r)=b_0\zeta_A(n)+\sum_{i=1}^{n-1}a_i\zeta_A(i,n-i)
\end{equation}
for some $a_i$, $b_0\in k$. We are interested in $n$-tuples of coefficients $\mathfrak{C}=(b_0,a_1,\dots,a_{n-1})\in k^n$ satisfying the equation \eqref{shueq}.

In this paper we provide a $t$-motivic interpretation of shuffle relations for multizeta values over $k$. Let $r$, $s$, $n$ be given as above. We construct Frobenius modules (see \S 2.2) $M'$ and $M_\mathfrak{C}$ associated with the given $n$-tuple of coefficients $\mathfrak{C}=(b_0,a_1,\dots,a_{n-1})\in k^n$, which fits into the short exact sequence of Frobenius modules\[ 0\rightarrow M'\rightarrow M_\mathfrak{C}\rightarrow \mathbf{1}\rightarrow 0\] so $M_\mathfrak{C}$ represents a class in $\Ext_{\mathscr{F}}^1(\mathbf{1},M').$ For more details, we refer readers to \S \ref{Fro} and \S \ref{Ext}. To determine whether the given $n$-tuple of coefficients $\mathfrak{C}=(b_0,a_1,\dots,a_{n-1})\in k^n$ satisfy a shuffle relation over $k$ as in the equation \eqref{shueq}, we relate it to whether $M_\mathfrak{C}$ is an $\mathbb{F}_q[t]$-torsion class in $\Ext_{\mathscr{F}}^1(\mathbf{1},M')$. More precisely, if the $n$-tuple of coefficients $\mathfrak{C}\in k^n$ satisfy the equation \eqref{shueq}, then we show that the Frobenius module $M_\mathfrak{C}$ represents an $\mathbb{F}_q[t]$-torsion class in $\Ext_{\mathscr{F}}^1(\mathbf{1},M')$. Conversely, if the Frobenius module $M_\mathfrak{C}$ represents an $\mathbb{F}_q[t]$-torsion class in $\Ext_{\mathscr{F}}^1(\mathbf{1},M')$, then the $n$-tuple of coefficients $\mathfrak{C}\in k^n$ satisfy the equation \eqref{shueq} in the case $(q-1)\nmid n$. In the case $(q-1)\mid n$, the $n$-tuple of coefficients $\mathfrak{C}\in k^n$ satisfy the equation \eqref{shueq} modulo $\tilde{\pi}^{n}$. We state the above results in the Theorem \ref{main}.

Following the strategy in \cite{2014arXiv1411.0124C}, we give an effective criterion for the $\mathbb{F}_q[t]$-torsion property of $M_\mathfrak{C}$. With the $t$-motivic interpretation of shuffle relations for multizeta values over $k$ given in the Theorem \ref{main}, we can effectively determine whether the given $n$-tuple of coefficients $\mathfrak{C}=(b_0,a_1,\dots,a_{n-1})\in k^n$ satisfy a shuffle relation over $k$ (resp. shuffle relation over $k$ modulo $\tilde{\pi}^n$) in the case $(q-1)\nmid n$ (resp. $(q-1)\mid n$). We also provide some examples in \S \ref{Algorithm}.

After we worked out this project, we found that there is another approach to determine a given $n$-tuple $\mathfrak{C}=(b_0,a_1,\dots,a_{n-1})\in k^n$ satisfying a shuffle relation over $k$ by using the results provided by Chang \cite{2015arXiv151006519C}. We fix $r,s\in\mathbb{N}$ and put $n:=r+s$. By combining a shuffle relation over $k$ and the relation (\ref{chen}), we have a relation of the form 
\begin{equation}\label{dou}
\widetilde{b}_0\zeta_A(n)+\sum_{i=1}^{n-1}\widetilde{a}_i\zeta_A(i,n-i)=0
\end{equation} where $\widetilde{a}_i$, $\widetilde{b}_0\in k$. So we have a one-to-one correspondence between shuffle relations over $k$ and relations over $k$ of the form (\ref{dou}). By \cite[Thm. 5.1.1, Thm. 6.1.1]{2015arXiv151006519C}, we have an effective process to check whether a given sequence $\widetilde{a}_i$, $\widetilde{b}_0\in k$ satisfies (\ref{dou}) in the case $(q-1)\nmid n$. In the case $(q-1)\mid n$, this effective process can check whether a given sequence $\widetilde{a}_i$, $\widetilde{b}_0\in k$ satisfies (\ref{dou}) modulo $\tilde{\pi}^n$. Hence we achieve the same result from this approach. For more details, we refer readers to \S \ref{Another}.

Comparing with the results provided in this paper, let us consider the analogue question in the classical case. For fixed positive integer $d$ and $d$-tuple of positive integer variable $(s_1,\dots,s_d)$ with $s_1>1$, \textit{the classical multiple zeta value} is defined by \[\zeta(s_1,\dots,s_d):=\sum_{k_1>\cdots>k_d>0}k_1^{-s_1}\cdots k_d^{-s_d}\text{ (see \cite{Z_MZV})}.\] We fix positive integers $r$, $s>1$ and let $n:=r+s$. There are two well-known formulas. One is shuffle product, also known as Euler's decomposition formula:

\[\zeta(r)\zeta(s)=\sum_{\substack{i\geq 2, j\geq 1\\ i+j=r+s}}\left[ \begin{pmatrix}i-1\\r-1\end{pmatrix}+\begin{pmatrix}j-1\\s-1\end{pmatrix}\right]\zeta(i,j),
\] and the other is stuffle product:
\[\zeta(r)\zeta(s)=\left(\sum_{n_1=n_2}+\sum_{n_1>n_2}+\sum_{n_1<n_2}\right)\frac{1}{n_1^rn_2^s}=\zeta(r+s)+\zeta(r,s)+\zeta(s,r).\] Note that we used to think the field of rational functions $k$ as an analogue of the field of relational numbers $\mathbb{Q}$. Studying $n$-tuples of coefficients $\mathfrak{C}=(b_0,a_1,\dots,a_{n-1})\in \mathbb{Q}^n$ satisfying the following equation: \begin{equation}\label{cshueq}
\zeta(r)\zeta(s)-\zeta(r,s)-\zeta(s,r)=b_0\zeta(n)+\sum_{i=2}^{n-1}a_i\zeta(i,n-i)
\end{equation} appeals to us, and we wonder if there is any criterion for a given $n$-tuple of rational numbers satisfying the equation \eqref{cshueq}. But in this case it is still an unknown problem.

The paper is organized as follows. In \S \ref{PPP}, we set up some essential preliminaries first, and then we state our main theorem, Theorem \ref{main}, which gives a $t$-motivic interpretation of shuffle relations. We prove our main theorem in \S \ref{pf} and provide a necessary condition for a shuffle relation in \S \ref{fur}. We state an effective criterion whether $M_\mathfrak{C}$ is $\mathbb{F}_q[t]$-torsion in $\Ext_{\mathscr{F}}^1(\mathbf{1},M')$ and write down an algorithm in \S \ref{eff}, \S \ref{Algorithm}, respectively. We also provide some examples in \S \ref{Algorithm}. In \S \ref{Another}, we give another approach to our result. The crucial property which makes our criterion effective is the identification of $\Ext_{\mathscr{F}}^1(\mathbf{1},M')$ as a $t$-module defined over $A$ in which $M_\mathfrak{C}$ corresponds to an integral point. However, its proof is essentially the same as \cite{2014arXiv1411.0124C} and so we leave the detailed proof in the appendix.

\renewcommand{\abstractname}{Acknowledgements}
\begin{abstract}
I wish to thank Prof. C.-Y. Chang for helpful advice and Prof. J. Yu for inspiring me to work on this project and enlightening me with constructive opinion. I am grateful to Prof. W. D. Brownawell for carefully reading my preprint and providing me many useful comments on writing. I thank Prof. M. Kaneko for the information about the classical issues of this project. I further thank Prof. M. A. Papanikolas for useful suggestions on Maple programming. I also thank Dr. H.-J. Chen, Dr. Y.-L. Kuan, Prof. D. Thakur and Prof. T.-Y. Wang for their valuable comments. Finally, I thank the referee for several useful suggestions.
\end{abstract}

\section{Preliminaries and The Main Theorems}\label{PPP}

\subsection{Some notations and definitions}
Let $\mathbb{F}_q$ be the finite field with $q$ elements, where $q$ is a power of a prime $p$. Let $\theta$ be a variable and $A:=\mathbb{F}_q[\theta]$, the polynomial ring in $\theta$ over $\mathbb{F}_q$. We denote by $A_+$ the set of monic polynomials in $A$. Let $k:=\mathbb{F}_q(\theta)$, the field of fractions of $A$, and define the absolute value $|\cdot|_\infty$ associated to the infinite place of $k$ so that $|\theta|_\infty =q$. Let $k_\infty$ be the completion of $k$ with respect to $|\cdot|_\infty$. Note that $k_\infty$ is equal to $\mathbb{F}_q(\!( 1/\theta)\!)$, the field of Laurent series in $1/\theta$ over $\mathbb{F}_q$. Let $\overline{k_\infty}$ be a fixed algebraic closure of $k_\infty$. We denote by $\overline{k}$ the algebraic closure of $k$ in $\overline{k_\infty}$, and let $\mathbb{C}_\infty$ be the completion of $\overline{k_\infty}$ with respect to the canonical extension of $|\cdot|_\infty$.

We recall the characteristic $p$ multizeta values defined by Thakur.
\begin{defi}[\cite{Thakurfunctionfield}]
For any $r$-tuple of positive integers $(s_1,\dots,s_r)\in \mathbb{N}^r$, we define \[\zeta_A(s_1,\dots,s_r):=\sum_{\substack{(a_1,\dots,a_r)\in A^r_+\\ \deg a_1>\cdots>\deg a_r}}\frac{1}{a_1^{s_1}\cdots a_r^{s_r}}\in k_\infty.\]
\end{defi}

\begin{rmk}\label{non}
In  \cite{Thakur09}, Thakur showed that each multizeta value is non-vanishing.
\end{rmk}

\subsection{Anderson-Thakur polynomials}

Define $D_0:=1$ and $D_i:=\prod_{j=0}^{i-1}(\theta^{q^i}-\theta^{q^j})$ for $i\in \mathbb{N}.$ For a non-negative integer $n$, we express $n$ as \[n=\sum_{i=0}^\infty n_i q^i \quad (0\leq n_i\leq q-1,~n_i=0~\text{for}~i\gg 0),\] and we recall the \textit{Carlitz factorial}\[\Gamma_{n+1}:=\prod_{i=0}^{\infty}D_i^{n_i}\in A\text{ (see \cite{Thakurfunctionfield})}.\]

We put $G_0(y):=1$ and define polynomials $G_n(y)\in \mathbb{F}_q[t,y]$ for $n\in \mathbb{N}$ by the product \[G_n(y):=\prod_{i=1}^n(t^{q^n}-y^{q^i})\]where $t$ is a new variable independent from $y$. For $n=0,1,2,\dots$, we define the sequence of \textit{Anderson-Thakur polynomials} $H_n \in A[t]$ by the generating function identity \[ \left( 1-\sum_{i=0}^\infty \frac{G_i(\theta)}{D_i|_{\theta=t}}x^{q^{i}}\right) ^{-1}=\sum_{n=0}^\infty \frac{H_n}{\Gamma_{n+1}|_{\theta=t}}x^n \text{ (see \cite{AT90,AT09})}.\]

\subsection{Frobenius twisting and Frobenius modules}\label{Fro}
We consider the \textit{Frobenius twisting} which is an automorphism on $\mathbb{C}_\infty(\!( t)\!)$ defined by \[\mathbb{C}_\infty(\!( t)\!)\rightarrow\mathbb{C}_\infty(\!( t)\!):~f:=\sum_ia_it^i\mapsto f^{(-1)}:=\sum_ia_i^{\frac{1}{q}}t^i.\] Note that the twisting is extended to $\Mat_n(\mathbb{C}_\infty(\!( t)\!))$ by acting entry-wisely.

We define $\bar{k}[t,\sigma]$ to be the non-commutative $\bar{k}[t]$-algebra generated by $\sigma$ with respect to the relation\[\sigma f=f^{(-1)}\sigma,~\forall f\in \bar{k}[t].\]
A left $\bar{k}[t,\sigma]$-module $M$ is called a \textit{Frobenius module} if it is free of finite rank over $\bar{k}[t]$. \textit{Morphisms of Frobenius modules} are left $\bar{k}[t,\sigma]$-module homomorphisms. We let $\mathscr{F}$ be the category of Frobenius modules.

We denote by $\mathbf{1}$ the trivial object in $\mathscr{F}$. Its underlying space is $\bar{k}[t]$ subject to the $\sigma$-action \[\sigma (f):=f^{(-1)},~\forall f\in\mathbf{1}.\]

Let $\Phi\in\Mat_r(\bar{k}[t])$ be given. We say that a Frobenius module $M$ is \textit{defined by the matrix $\Phi\in\Mat_r(\bar{k}[t])$} if the Frobenius module $M$ is of rank $r$ over $\bar{k}[t]$ with  $\bar{k}[t]$-basis $\{f_1,\dots,f_r\}\subset M$ satisfying 
\[\sigma\begin{pmatrix}
 f_1\\ 
 \vdots\\
 f_r
\end{pmatrix}
:=\begin{pmatrix}
\sigma (f_1)\\ 
\vdots\\
\sigma (f_r)
\end{pmatrix}
=\Phi\begin{pmatrix}
 f_1\\ 
 \vdots\\
 f_r
\end{pmatrix}.\]

\subsection{$\Ext^1$-modules}\label{Ext}
Let $M$, $M'$ be two objects in $\mathscr{F}$ defined by two matrices $\Phi$, $\Phi'$ respectively for which $\Phi=\begin{pmatrix}
 \Phi'& 0\\ 
 \mathbf{v}& 1
\end{pmatrix}\in\Mat_r(\bar{k}[t])$ $(r\geq 2)$ for some row vector $\mathbf{v}$. Since $M$ fits into the short exact sequence of Frobenius modules \[ 0\rightarrow M'\rightarrow M\rightarrow \mathbf{1}\rightarrow 0,\] $M$ represents a class in $\Ext_{\mathscr{F}}^1(\mathbf{1},M').$

The set $\Ext_{\mathscr{F}}^1(\mathbf{1},M')$ forms a group under the Baer sum. Furthermore, it has an $\mathbb{F}_q[t]$-module structure given by the following. Given $M_1$, $M_2$ representing classes in $\Ext_{\mathscr{F}}^1(\mathbf{1},M')$ defined by two matrices $\Phi_1$, $\Phi_2\in\Mat_r(\bar{k}[t])$ respectively, we write\[\Phi_1=\begin{pmatrix}
 \Phi'& 0\\ 
 \mathbf{v_1}& 1
\end{pmatrix},~ \Phi_2=\begin{pmatrix}
 \Phi'& 0\\ 
 \mathbf{v_2}& 1
\end{pmatrix}.\] Then the Baer sum of the two classes of $M_1$, $M_2$ is the class of the object $M_1 \underset{B}{+}M_2\in\mathscr{F}$ defined by the matrix \[\begin{pmatrix}
 \Phi'& 0\\ 
 \mathbf{v_1}+\mathbf{v_2}& 1
\end{pmatrix}\in\Mat_r(\bar{k}[t]).\] Given any $a\in\mathbb
{F}_q[t]$, the action of $a\in\mathbb{F}_q[t]$ on the class of $M_1$ is the class of the object $a\ast M_1\in\mathscr{F}$ defined by the matrix \[\begin{pmatrix}
 \Phi'& 0\\ 
a \mathbf{v_1}& 1
\end{pmatrix}\in\Mat_r(\bar{k}[t]).\]

\subsection{The Main Theorem}\label{notation}
\begin{defi}
We fix positive integers $r,s\in\mathbb{N}$ and let $n:=r+s$ be given. An $n$-tuple $\mathfrak{C}=(b_0,a_1,\dots,a_{n-1})\in k^{n}$ is said to have the \textit{\ref{shu}-property} if the following equation holds:
\begin{equation}\label{shu}\tag{SR}
\zeta_A(r)\zeta_A(s)-\zeta_A(r,s)-\zeta_A(s,r)=b_0\zeta_A(n)+\sum_{i=1}^{n-1}a_i\zeta_A(i,n-i). 
\end{equation}

\end{defi}

\begin{rmk}The equations of the above form, called shuffle relations by Thakur, were first studied by Thakur in \cite{Thakurshuffle}, and he proved the existence of $n$-tuples in $\mathbb{F}_p^n\subset k^n$ having the \ref{shu}-property in the same paper. Chen \cite{Chen2015153} gave an explicit $n$-tuple in $\mathbb{F}_p^n\subset k^n$ having the \ref{shu}-property.
\end{rmk}

\begin{defi}\label{ab}
Given $\mathfrak{C}=(b_0,a_1,\dots,a_{n-1})\in k^n$ and letting $\Gamma_\mathfrak{C}\in A$ be the monic least common multiple of the denominators of $\{\frac{a_i}{\Gamma_i\Gamma_{n-i}}:i=1,\dots, n-1\}\cup\{\frac{b_0}{\Gamma_n},~\frac{1}{\Gamma_r\Gamma_{s}}\}$, we put \[\alpha_i:=\frac{a_i\Gamma_\mathfrak{C}}{\Gamma_i\Gamma_{n-i}}|_{\theta=t}\in \mathbb{F}_q[t],~\beta_0:=\frac{b_0\Gamma_\mathfrak{C}}{\Gamma_n}|_{\theta=t}\in \mathbb{F}_q[t],~\gamma_0:=\frac{\Gamma_\mathfrak{C}}{\Gamma_r\Gamma_{s}}|_{\theta=t}\in \mathbb{F}_q[t]\]
for each $i$. Then we define the \textit{associated Frobenius module $M_\mathfrak{C}$ of $\mathfrak{C}$} which is defined by the matrix $\Phi_\mathfrak{C}\in \Mat_{n+1}(\overline{k}[t])$:
\[\Phi_\mathfrak{C} =
\begin{pmatrix}
(t-\theta)^{n}& & & &\\ 
 H^{(-1)}_{1-1}(t-\theta)^n& (t-\theta)^{n-1}& & & \\
\vdots  & & \ddots& &\\
H^{(-1)}_{(n-1)-1}(t-\theta)^n & & & (t-\theta)&\\
{\Phi_\mathfrak{C}}_{(n+1),1}& \alpha_1H^{(-1)}_{(n-1)-1}(t-\theta)^{n-1}& \cdots& \alpha_{n-1}H^{(-1)}_{1-1}(t-\theta)^1&1
\end{pmatrix},\]
where \[{\Phi_\mathfrak{C}}_{(n+1),1}:=\beta_0H^{(-1)}_{n-1}(t-\theta)^n-\gamma_0H^{(-1)}_{r-1}H^{(-1)}_{s-1}(t-\theta)^{n}.\]
\end{defi}

Let $\Phi'\in \Mat_{n}(\overline{k}[t])$ be the square matrix of size $n$ in the upper left-hand corner of $\Phi_\mathfrak{C}$, \textit{i.e.}\begin{equation}\label{phi}\Phi'=\begin{pmatrix}
(t-\theta)^{n}& & & \\ 
 H^{(-1)}_{1-1}(t-\theta)^n& (t-\theta)^{n-1}& & \\
\vdots  & & \ddots& \\
H^{(-1)}_{(n-1)-1}(t-\theta)^n & & & (t-\theta)
\end{pmatrix},\end{equation} and let $M'$ be the Frobenius module defined by $\Phi'$. Note that the Frobenius module $M_\mathfrak{C}$ represents a class in $\Ext_{\mathscr{F}}^1(\mathbf{1},M')$. 

If the $n$-tuple $\mathfrak{C}$ has the \ref{shu}-property, we show that the corresponding Frobenius module $M_\mathfrak{C}$ represents an $\mathbb{F}_q[t]$-torsion class in $\Ext_{\mathscr{F}}^1(\mathbf{1},M')$. Conversely, if $M_\mathfrak{C}$ represents an $\mathbb{F}_q[t]$-torsion class in $\Ext_{\mathscr{F}}^1(\mathbf{1},M')$, it is natural to ask if the $n$-tuple $\mathfrak{C}$ has the \ref{shu}-property.

Our main result is stated as follows.

\begin{thm}\label{main}
Let $\mathfrak{C}=(b_0,a_1,\dots,a_{n-1})\in k^n$ be given.
\begin{enumerate}[(1)]
\item If the $n$-tuple $\mathfrak{C}$ has the \ref{shu}-property, then the Frobenius module $M_\mathfrak{C}$ represents an $\mathbb{F}_q[t]$-torsion class in $\Ext_{\mathscr{F}}^1(\mathbf{1},M')$.\label{main1}
\item Suppose that the Frobenius module $M_\mathfrak{C}$ represents an $\mathbb{F}_q[t]$-torsion class in $\Ext_{\mathscr{F}}^1(\mathbf{1},M')$.\label{main2}
\begin{enumerate}
\item If $(q-1)\nmid n$, then the $n$-tuple $\mathfrak{C}$ has the \ref{shu}-property.
\item If $(q-1)\mid n$, then there exists unique $\widetilde{b}_0\in k$ such that the $n$-tuple $\widetilde{
\mathfrak{C}}=(\widetilde{b}_0,a_1,\dots,a_{n-1})$ has the \ref{shu}-property.\label{main22}
\end{enumerate}
\end{enumerate}
\end{thm}

\section{Proof of the Theorem \ref{main}}\label{pf}
\subsection{Some important properties}
\begin{defi}[\cite{ABP04}]
A formal power series $\sum_{n=0}^\infty a_nt^n\in \bar{k}[\![ t]\!]$ is called \textit{entire} if \[\lim_{n\to\infty}\sqrt[n]{|a_n|_\infty}=0,\text{ and }[k_\infty(a_0,a_1,\dots):k_\infty]<\infty.\]The set of all entire functions is denoted by $\mathscr{E}.$
\end{defi}
We fix a fundamental period $\tilde{\pi}$ of the Carlitz module $\mathcal{C}$ (see \cite{Goss96,Thakurfunctionfield}), and define \[\Omega (t):=(-\theta)^{\frac{-q}{q-1}}\prod_{i=1}^\infty\left( 1-\frac{t}{\theta^{q^i}}\right)\in \mathbb{C}_\infty[\![ t]\!],\] where $(-\theta)^{\frac{-1}{q-1}}$ is a choice of $(q-1)$th root of $-\theta$ so that $\frac{1}{\Omega (\theta)}=\tilde{\pi}$. Note that the power series is entire and we have the functional equation $\Omega^{(-1)}(t)=(t-\theta)\Omega(t)$ (see \cite{ABP04}).

The following are important properties developed by Anderson and Thakur (see \cite{AT90,AT09}): 
\begin{equation}\label{ppp}
(\Omega^sH_{s-1})^{(d)}(\theta)=\frac{\Gamma_s S_d(s)}{\tilde{\pi}^s},\quad\forall s\in\mathbb{N},~d\in\mathbb{Z}_{\geq 0},
\end{equation} where $S_d(s)$ is the power sum \[S_d(s):=\sum_{\substack{a\in A_+\\ \deg_\theta a=d}}\frac{1}{a^s}\in k.\] Furthermore, if we view $H_n$ as a polynomial in $\mathbb{F}_q[t][\theta]$, then we also have \begin{equation}\label{deg}
\deg_\theta H_n\leq\frac{nq}{q-1}.
\end{equation}

Given an $r$-tuple of positive integers $\mathfrak{s}=(s_1,\dots,s_r)$ and let $\mathfrak{Q}$ be the $r$-tuple of Anderson-Thakur polynomials $\mathfrak{Q}:=(H_{s_1-1},\dots,H_{s_r-1})$, we define the series \[\mathscr{L}_{\mathfrak{s},\mathfrak{Q}}:=\sum_{i_1>\cdots >i_r\geq0}(\Omega^{s_r}H_{s_r-1})^{(i_r)}\cdots (\Omega^{s_1}H_{s_1-1})^{(i_1)}\in\mathbb{C}_\infty[\![ t]\!]\text{ (cf. \cite{AT09})}.\] Since $\Omega$ satisfies the functional equation $\Omega^{(1)} (t)=\frac{\Omega (t)}{(t-\theta)}$, we have \[\mathscr{L}_{\mathfrak{s},\mathfrak{Q}}=\Omega^{s_1+\cdots +s_r}\sum_{i_1>\cdots >i_r\geq0}\frac{H_{s_r-1}^{(i_r)}(t)\cdots H_{s_1-1}^{(i_1)}(t)}{[(t-\theta^q)\cdots(t-\theta^{q^{i_r}})]^{s_r}\cdots[(t-\theta^q)\cdots(t-\theta^{q^{i_1}})]^{s_1}}.\] By \eqref{deg}, the series $\mathscr{L}_{\mathfrak{s},\mathfrak{Q}}$ is in the Tate algebra $\mathbb{T}$, where \[\mathbb{T}:=\{f\in\mathbb{C}_\infty[\![ t]\!]:f\text{ converges on }|t|_\infty\leq 1\}.\]

For $1\leq \ell<j\leq r+1$, we define the series\[\mathscr{L}_{j,\ell}:=\sum_{i_\ell>\cdots >i_{j-1}\geq0}(\Omega^{s_{j-1}}H_{s_{j-1}-1})^{(i_{j-1})}\cdots(\Omega^{s_\ell}H_{s_\ell-1})^{(i_\ell)}\in\mathbb{C}_\infty[\![ t]\!]. \] Note that we have $\mathscr{L}_{\mathfrak{s},\mathfrak{Q}}=\mathscr{L}_{r+1,1}$, and (\ref{ppp}) gives \[\mathscr{L}_{r+1,1}(\theta)=\tilde{\pi}^{-(s_1+\cdots+s_r)}\Gamma_{s_1}\cdots\Gamma_{s_r}\zeta_A(s_1,\dots,s_r).\]

\begin{rmk}\label{remm} 
We also have the following properties:
\begin{enumerate}[(1)]
\item Chang \cite[Lem. 5.3.1]{Chang14} showed that $\mathscr{L}_{j,\ell}$ is actually an entire function for all $\ell,$ $j$ with $1\leq \ell<j\leq r+1$.

\item By \cite[Prop. 2.3.3]{2014arXiv1411.0124C}, we have for $1\leq \ell<j\leq r+1$, \[\label{prop}
\mathscr{L}_{j,\ell}(\theta^{q^N})=\mathscr{L}_{j,\ell}(\theta)^{q^N}
\text{ for all }N\in\mathbb{N}.\]

\item The equation (\ref{ppp}) and \textit{Remark} \ref{non} give that $\mathscr{L}_{j,\ell}$ is non-vanishing at $\theta^{q^N}$ for all $\ell,$ $j$ with $1\leq \ell<j\leq r+1$, $N\in\mathbb{Z}_{\geq 0}$.\label{nonvanished}
\end{enumerate}
\end{rmk}

\subsection{A key lemma}

\begin{lem}\label{zero}
Let $\{s_i\}_{i=1}^I\subseteq \mathbb{Z}_{\geq 0}$ be a strictly increasing finite sequence, and let $\{L_i\}_{i=1}^{I}\subseteq\mathbb{T}$ satisfying \[L_i(\theta^{q^N})\neq 0 \] for all $N\in \mathbb{N}\cup\{0\}$, $i \in\{1,\dots,I\}$ be given. For any $\{B_i\}_{i=1}^I\subseteq \bar{k}(t)$ satisfying \begin{equation}\label{l}
\sum_{i=1}^IB_i\Omega^{s_i}L_i=0,
\end{equation} we have $B_i=0$ for all $i \in\{1,\dots,I\}.$
\end{lem}

\begin{proof}
(cf. the proof of \cite[Thm. 2.5.2]{2014arXiv1411.0124C}, \cite[Thm. 3.1.1]{2015arXiv151006519C}) First, we divide the equation (\ref{l}) by $\Omega^{s_1}$. Then it becomes \begin{equation}\label{iii}
B_1L_1+\sum_{i=2}^IB_i\Omega^{\widetilde{s}_i}L_i=0
\end{equation} where $\widetilde{s}_i =s_i-s_1$. Note that each $B_i$ is defined at $\theta^{q^N}$ for sufficiently large $N\in \mathbb{N}$ since $B_i$ belongs to $\bar{k}(t)$. Also note that $\Omega$ has a simple zero at $\theta^{q^N}$ for each $N\in \mathbb{N}$, and hence (\ref{iii}) gives rise to \[B_1(\theta^{q^N})L_1(\theta^{q^N})=0\] for sufficiently large $N\in \mathbb{N}.$ By the assumption $L_1(\theta^{q^N})\neq 0$ for all  $N\in \mathbb{N},$ which implies that \[B_1(\theta^{q^N})=0\] for all large $N\in\mathbb{N}$, whence $B_1=0$. 

The equation (\ref{l}) becomes \[\sum_{i=2}^IB_i\Omega^{\widetilde{s}_i}L_i=0.\] Since  $\{\widetilde{s}_i\}_{i=2}^I$ is also a strictly increasing finite sequence, we repeat the same process above and then conclude that $B_i=0$ for all $i\in\{1,\dots,I\}.$
\end{proof}

\subsection{Proof of the Theorem \ref{main}(\ref{main1})}
(cf. the proof of \cite[Thm. 2.5.2]{2014arXiv1411.0124C}, \cite[Thm. 3.1.1]{2015arXiv151006519C}) 

For $i=1,\dots,n-1,$ we define two series as follows:
\begin{align*}
\mathscr{L}^{[i,n-i]}& :=\sum_{\ell_1>\ell_2\geq 0}\left(\Omega^{n-i}H_{(n-i)-1}\right)^{(\ell_2)}\left(\Omega^{i}H_{i-1}\right)^{(\ell_1)}, \\
\mathscr{L}^{[i]}& :=\sum_{\ell\geq 0}\left(\Omega^iH_{i-1}\right)^{(\ell)}.\\
\end{align*}
Let $\psi_\mathfrak{C}\in\Mat_{(n+1)\times 1}(\mathscr{E})$ be defined by
\[\psi_\mathfrak{C}=
\begin{pmatrix}
\Omega^n\\
\Omega^{n-1}\mathscr{L}^{[1]}\\
\vdots\\
\Omega\mathscr{L}^{[n-1]}\\
\beta_0\mathscr{L}^{[n]}+\sum_{i=1}^{n-1}\alpha_i\mathscr{L}^{[i,n-i]}-\gamma_0(\mathscr{L}^{[r]}\mathscr{L}^{[s]}-\mathscr{L}^{[r,s]}-\mathscr{L}^{[s,r]})
\end{pmatrix}.
\]
Then we have $\psi_\mathfrak{C}^{(-1)}=\Phi_\mathfrak{C}\psi_\mathfrak{C}$ and \[\psi_\mathfrak{C}(\theta)=\tilde{\pi}^{-n}
\begin{pmatrix}
1\\
\Gamma_{1}\zeta_A(1)\\
\vdots\\
\Gamma_{n-1}\zeta_A(n-1)\\
\Gamma_\mathfrak{C}[b_0\zeta_A(n)+\sum_{i=1}^{n-1}a_i\zeta_A(i,n-i)-\left( \zeta_A(r)\zeta_A(s)-\zeta_A(r,s)-\zeta_A(s,r)\right)]
\end{pmatrix}\]is in $\Mat_{(n+1)\times 1}(\overline{k_\infty})$.

Let $\mathbf{v}=(0,\dots,0,1)\in \Mat_{1\times (n+1)}(\overline{k})$. Then the the \ref{shu}-property gives $\mathbf{v}\psi_\mathfrak{C}(\theta)=0$. Note that the above satisfies the assumption of \cite[Thm. 3.1.1]{ABP04}. By \cite[Thm. 3.1.1]{ABP04}, there exists \[\mathbf{f}=(f_1,\dots,f_{n+1})\in \Mat_{1\times (n+1)}(\bar{k}[t])\] such that $\mathbf{f}\psi_\mathfrak{C}=0$ and $\mathbf{f}(\theta)=\mathbf{v}.$ Now we put $\widetilde{\mathbf{f}}=\frac{1}{f_{n+1}}\mathbf{f} \in\Mat_{1\times (n+1)}(\bar{k}(t))$ and note that $\widetilde{\mathbf{f}}$ is regular at $t=\theta $. We claim that
\begin{equation}\label{claim1}
\widetilde{\mathbf{f}}-\widetilde{\mathbf{f}}^{(-1)}\Phi_\mathfrak{C}=(0,\dots,0).
\end{equation}

Let us assume this claim first. Then we have the following equation
\begin{equation}\label{hom1}
\begin{pmatrix}
 1& &  & \\
 &  \ddots&  & \\ 
 &  & 1& \\ 
 \frac{f_1}{f_{n+1}}& \cdots &\frac{f_{n}}{f_{n+1}}  & 1
\end{pmatrix}^{(-1)}\Phi_\mathfrak{C}=
\begin{pmatrix}
 \Phi'& \\ 
 & 1
\end{pmatrix}
\begin{pmatrix}
 1& &  & \\
 &  \ddots&  & \\ 
 &  & 1& \\ 
 \frac{f_1}{f_{n+1}}& \cdots &\frac{f_{n}}{f_{n+1}}  & 1
\end{pmatrix}.
\end{equation}
The equation (\ref{hom1}) gives a left $\bar{k}(t)[\sigma]$-module homomorphism between $\bar{k}(t)\underset{\bar{k}[t]}{\otimes}(M'\oplus\mathbf{1})$ and $\bar{k}(t)\underset{\bar{k}[t]}{\otimes}M_\mathfrak{C}$. By \cite[Prop. 2.2.1]{2014arXiv1411.0124C}, the common denominator of $\frac{f_1}{f_{n+1}},\dots,\frac{f_{n}}{f_{n+1}}$, say $c$, is in $\mathbb{F}_q[t]$. Write $\Phi_\mathfrak{C}=\begin{pmatrix}
 \Phi'& \\ 
\mu & 1
\end{pmatrix}$ for some $\mu\in \Mat_{1\times (n+1)}(\bar{k}[t]),$ and put $c\ast \Phi_\mathfrak{C}:=\begin{pmatrix}
 \Phi'& \\ 
c\mu & 1
\end{pmatrix},$ which defines the Frobenius module $c\ast M_\mathfrak{C}$ representing a class in $\Ext_{\mathscr{F}}^1(\mathbf{1},M')$. The class $c\ast M_\mathfrak{C}$ represents the trivial class in $\Ext_{\mathscr{F}}^1(\mathbf{1},M')$ since we have the equation 
\begin{equation}
\begin{pmatrix}
 1& &  & \\
 &  \ddots&  & \\ 
 &  & 1& \\ 
 \delta_1& \cdots &\delta_{n+1}  & 1
\end{pmatrix}^{(-1)}c\ast\Phi_\mathfrak{C}=
\begin{pmatrix}
 \Phi'& \\ 
 & 1
\end{pmatrix}
\begin{pmatrix}
 1& &  & \\
 &  \ddots&  & \\ 
 &  & 1& \\ 
 \delta_1& \cdots &\delta_{n+1}  & 1
\end{pmatrix},
\end{equation} where $\delta_i=c\frac{f_i}{f_{n+1}}\in \bar{k}[t]$ for $i=1,\dots,n+1,$ \textit{i.e.}, the class of $M_\mathfrak{C}$ is $c$-torsion in the $\mathbb{F}_q[t]$-module $\Ext_{\mathscr{F}}^1(\mathbf{1},M')$.

To complete the proof, we need to verify the equation (\ref{claim1}). Applying the Frobenius twisting $(\cdot)^{(-1)}$ on the equation $\widetilde{\mathbf{f}}\psi_\mathfrak{C}=0$ and subtracting it from the equation $\widetilde{\mathbf{f}}\psi_\mathfrak{C}=0$, we have \[\left( \widetilde{\mathbf{f}}-\widetilde{\mathbf{f}}^{(-1)}\Phi_\mathfrak{C}\right)\psi_\mathfrak{C}=0.\] Let $(B_1,\dots,B_{n+1}):=\widetilde{\mathbf{f}}-\widetilde{\mathbf{f}}^{(-1)}\Phi$. Note that $B_{n+1}=0$ and the above equation becomes 
\begin{equation}\label{hypo1}
 \sum_{i=1}^{n}B_i\Omega^{n-i+1}\mathscr{L}^{[i-1]}=0
,\end{equation}where we define $\mathscr{L}^{[0]}=1$ for convenience. By \textit{Remark} \ref{remm}(\ref{nonvanished}), the equation (\ref{hypo1}) satisfies the hypothesis of Lemma \ref{zero}. It follows by Lemma \ref{zero} that $B_i=0$ for all $i=1,\dots,n+1,$ and so we complete the proof.

\subsection{Proof of the Theorem \ref{main}(\ref{main2})}\label{nota}
(cf. the proof of \cite[Thm. 2.5.2]{2014arXiv1411.0124C}) 

Suppose the Frobenius module $M_\mathfrak{C}$ represents an $\mathbb{F}_q[t]$-torsion class in $\Ext_{\mathscr{F}}^1(\mathbf{1},M'),$ \textit{i.e.}, there exists $c\in\mathbb{F}_q[t]\setminus \{ 0\}$ such that $c\ast M_\mathfrak{C}$ represents a trivial class in $\Ext_{\mathscr{F}}^1(\mathbf{1},M').$

Note that $c\ast M_\mathfrak{C}$ is defined by the matrix $X$ given as follows:

\[X=
\begin{pmatrix}
(t-\theta)^{n}& & & &\\ 
 H^{(-1)}_{1-1}(t-\theta)^n& (t-\theta)^{n-1}& & &\\
\vdots  & & \ddots& &\\
H^{(-1)}_{(n-1)-1}(t-\theta)^n & & & (t-\theta)&\\
X_{(n+1),1}& c\alpha_1H^{(-1)}_{(n-1)-1}(t-\theta)^{n-1}& \cdots& c\alpha_{n-1}H^{(-1)}_{1-1}(t-\theta)^1& 1
\end{pmatrix}\] where \[X_{(n+1),1}=c\left[ \beta_0H^{(-1)}_{n-1}(t-\theta)^n-\gamma_0H^{(-1)}_{r-1}H^{(-1)}_{s-1}(t-\theta)^{n}\right],\]and there exists $\delta_1,\dots,\delta_{n}\in \bar{k}[t]$ such that
\[\begin{pmatrix}
 1& & &\\ 
 & \ddots& &\\
& & 1& \\
 \delta_1& \dots& \delta_{n}& 1\\
\end{pmatrix}^{(-1)}X=\begin{pmatrix}
\Phi'&\\ 
&1\\
\end{pmatrix}\begin{pmatrix}
 1& & &\\ 
 & \ddots& &\\
& & 1& \\
 \delta_1& \dots& \delta_{n}& 1\\
\end{pmatrix}.\]

Consider 
\[Y=
\begin{pmatrix}
\Omega^n & & & & &\\
\Omega^{n-1}\mathscr{L}^{[1]}& \Omega^{n-1}& & & &\\
\Omega^{n-2}\mathscr{L}^{[2]}& \Omega^{n-2}& \Omega^{n-2}& & &\\
\vdots& \vdots& \vdots& \ddots& &\\
\Omega^{1}\mathscr{L}^{[n-1]}& \Omega^{1}& \Omega^{1}& \dots& \Omega^{1}&\\
Y_{n+1,1}& c\sum_{i=1}^{n-1}\alpha_i\mathscr{L}^{[n-i]}& c\sum_{i=2}^{n-1}\alpha_i\mathscr{L}^{[n-i]}& \dots& c\alpha_{n-1}\mathscr{L}^{[1]}& 1
\end{pmatrix}\]
where\[Y_{n+1,1}=c\left[\beta_0\mathscr{L}^{[n]}+\sum_{i=1}^{n-1}\alpha_i\mathscr{L}^{[i,n-i]}-\gamma_0(\mathscr{L}^{[r]}\mathscr{L}^{[s]}-\mathscr{L}^{[r]}-\mathscr{L}^{[s]})\right],\] then we have the relation $Y^{(-1)}=XY$. Putting $\mathcal{D}= (\delta_1,\dots\delta_{n})$ and $Y'=\begin{pmatrix}
 I_{n}&\\ 
\mathcal{D}& 1\\
\end{pmatrix} Y$, we have $Y'^{(-1)}=\begin{pmatrix}
\Phi'&\\ 
&1\\
\end{pmatrix}Y'$. Let $\Psi'$ be the square matrix of size $n$ cut from the upper left square of $Y'$. Since we also have $\begin{pmatrix}
\Psi'&\\ 
&1\\
\end{pmatrix}^{(-1)}=\begin{pmatrix}
\Phi'&\\ 
&1\\
\end{pmatrix}\begin{pmatrix}
\Psi'&\\ 
&1\\
\end{pmatrix}$, by \cite[\S 4.1.6]{P08}, there exists $\bm{\nu}=(\nu_1,\dots,\nu_{n})\in \mathbb{F}_q(t)^{n}$ such that \[Y'=\begin{pmatrix}
\Psi'&\\ 
&1\\
\end{pmatrix}\begin{pmatrix}
I_{n}&\\ 
\bm{\nu}&1\\
\end{pmatrix},\]\textit{i.e.},\[\begin{pmatrix}
 I_{n}&\\ 
\mathcal{D}& 1\\
\end{pmatrix} Y=\begin{pmatrix}
\Psi'&\\ 
&1\\
\end{pmatrix}\begin{pmatrix}
I_{n}&\\ 
\bm{\nu}&1\\
\end{pmatrix}.\] Therefore, we have
\begin{align*}
\nu_1={}& \sum_{i=1}^n\delta_i\Omega^{n-i+1}\mathscr{L}^{[i]}+c\left[\beta_0\mathscr{L}^{[n]}+\sum_{i=1}^{n-1}\alpha_i\mathscr{L}^{[i,n-i]}-\gamma_0(\mathscr{L}^{[r]}\mathscr{L}^{[s]}-\mathscr{L}^{[r]}-\mathscr{L}^{[s]})\right];\\
\nu_2={}& \sum_{i=2}^n\delta_i\Omega^{n-i+1}+c\sum_{i=1}^{n-1}\alpha_i\mathscr{L}^{[n-i]};\\
\vdots&\\
\nu_{n}={}& \sum_{i=n}^n\delta_i\Omega^{n-i+1}+c\sum_{i=n-1}^{n-1}\alpha_i\mathscr{L}^{[n-i]}(=\delta_n\Omega+c\alpha_{n-1}\mathscr{L}^{[1]});\\
\nu_{n+1}={}& \delta_{n+1}.
\end{align*}

We find that each $\nu_i$ is in $\mathbb{F}_q[t]$ since the right hand side of each equality above is in $\mathbb{T}.$ Now we evaluate $t=\theta^q$ in each equation above. Note that we work in fields with characteristic $p$ and $\Omega$ has a simple zero at $\theta^q$ . So by \textit{Remark} \ref{remm}(\ref{prop}), we get
\begin{align}\begin{split}
\nu_1(\theta)^q={}& (c(\theta)\Gamma_\mathfrak{C})^q\tilde{\pi}^{-nq}\left[ b_0\zeta_A(n)+\sum_{i=1}^{n-1}a_i\zeta_A(i,n-i)-(\zeta_A(r)\zeta_A(s)-\zeta_A(r,s)-\zeta_A(s,r))\right]^q;\\
\nu_2(\theta)^q={}& \left[c(\theta)\sum_{i=1}^{n-1}a_i\tilde{\pi}^{-(n-i)}\zeta_A(n-i)\right]^q;\\
\vdots&\\
\nu_{n}(\theta)^q={}& \left[c(\theta)\sum_{i=n-1}^{n-1}a_i\tilde{\pi}^{-(n-i)}\zeta_A(n-i)\right]^q(=\left[c(\theta)a_{n-1}\tilde{\pi}^{-1}\zeta_A(1)\right]^q).
\end{split}\end{align}
Taking the $q$th root of the both sides for each equation, we have

\begin{equation}\label{v1}
\nu_1(\theta)= c(\theta)\Gamma_\mathfrak{C}\tilde{\pi}^{-n}\left[ b_0\zeta_A(n)+\sum_{i=1}^{n-1}a_i\zeta_A(i,n-i)-(\zeta_A(r)\zeta_A(s)-\zeta_A(r,s)-\zeta_A(s,r))\right] ,
\end{equation}
and
\begin{align*}
\nu_2(\theta)={}& c(\theta)\sum_{i=1}^{n-1}a_i\tilde{\pi}^{-(n-i)}\zeta_A(n-i);\\
\vdots&\\
\nu_{n}(\theta)={}& c(\theta)\sum_{i=n-1}^{n-1}a_i\tilde{\pi}^{-(n-i)}\zeta_A(n-i)(=c(\theta)a_{n-1}\tilde{\pi}^{-1}\zeta_A(1));\\
\nu_{n+1}(\theta)={}& \delta_{n+1}(\theta).
\end{align*}

From (\ref{v1}), we have
\begin{equation}\label{re}
\frac{\nu_1(\theta)}{c(\theta)\Gamma_\mathfrak{C}}\tilde{\pi}^{n}=  b_0\zeta_A(n)+\sum_{i=1}^{n-1}a_i\zeta_A(i,n-i)-(\zeta_A(r)\zeta_A(s)-\zeta_A(r,s)-\zeta_A(s,r)).
\end{equation}

Note that if $(q-1)\nmid n$, then $\tilde{\pi}^{n}\notin k_\infty$. Since both $\frac{\nu_1(\theta)}{c(\theta)\Gamma_\mathfrak{C}}$ and the right hand side of (\ref{re}) are in $k_\infty$, we conclude that 
\[b_0\zeta_A(n)+\sum_{i=1}^{n-1}a_i\zeta_A(i,n-i)-(\zeta_A(r)\zeta_A(s)-\zeta_A(r,s)-\zeta_A(s,r))=0.
\]\textit{i.e.},\[\zeta_A(r)\zeta_A(s)-\zeta_A(r,s)-\zeta_A(s,r)=b_0\zeta_A(n)+\sum_{i=1}^{n-1}a_i\zeta_A(i,n-i).\]

If $(q-1)\mid n$, then by \cite{Car35} we know that \[\frac{\nu_1(\theta)}{c(\theta)\Gamma_\mathfrak{C}}\tilde{\pi}^{n}=b\zeta_A(n)\text{ for some }b\in k.\] We conclude that 
\[\zeta_A(r)\zeta_A(s)-\zeta_A(r,s)-\zeta_A(s,r)=\widetilde{b}\zeta_A(n)+\sum_{i=1}^{n-1}a_i\zeta_A(i,n-i),\] where $\widetilde{b}=b_0-b$. Note that the uniqueness of $\widetilde{b}$ is simply a consequence of \textit{Remark} \ref{non}.

\section{A necessary condition for the \ref{shu}-property}\label{fur}
Suppose the $n$-tuple $\mathfrak{C}=(b_0,a_1,\dots,a_{n-1})\in k^n$ has the \ref{shu}-property. Combining the shuffle relation with the relation (\ref{chen}) proved by Chen, we have a relation of the form \[b\zeta_A(n)+\sum_{i=1}^{n-1}a_i\zeta_A(i,n-i)=0.\]By \cite[Thm. 3.1.1]{2015arXiv151006519C}, we derive the following necessary condition for shuffle relations.

\begin{thm}
If an $n$-tuple $\mathfrak{C}=(b_0,a_1,\dots,a_{n-1})\in k^n$ has the \ref{shu}-property, then we have \[a_i=0\text{ if }(q-1)\nmid (n-i).\]
\end{thm}

\section{An effective criterion for the $\mathbb{F}_q[t]$-torsion property of $M_\mathfrak{C}$ in $\Ext_{\mathscr{F}}^1(\mathbf{1},M')$}\label{eff}
In this section, we will provide an effective criterion whether $M_\mathfrak{C}$ is $\mathbb{F}_q[t]$-torsion in $\Ext_{\mathscr{F}}^1(\mathbf{1},M').$ We follow the same idea in \cite[\S 5, \S 6]{2014arXiv1411.0124C}.

\subsection{Anderson $t$-modules}

Let $\tau :\mathbb{C}_\infty\to\mathbb{C}_\infty$ be the $q$th power operator defined by $x\mapsto x^q$ and let $\mathbb{C}_\infty[\tau]$ be the twisted polynomial ring in $\tau$
over $\mathbb{C}_\infty$ subject to the relation $\tau \alpha=\alpha^q\tau$ for $\alpha\in\mathbb{C}_\infty$. We define $t$-modules as follows.

\begin{defi}[\cite{A86}]
Let $d\in\mathbb{N}$ be given, a \textit{$d$-dimensional $t$-module} is a pair $(E,\phi)$, where $E$ is the $d$-dimensional algebraic group $\mathbb{G}^d_a$ and $\phi$ is an $\mathbb{F}_q[t]$-linear ring homomorphism \[\phi:\mathbb{F}_q[t]\to \Mat_d(\mathbb{C}_\infty[\tau])\]so that the image of $t$, denoted by $\phi_t$, is of the form $\alpha_0+\sum_i\alpha_i\tau^i$ with $\alpha_i\in \Mat_d(\mathbb{C}_\infty)$, and $\alpha_0-\theta I_d$ is a nilpotent matrix.
\end{defi}

\begin{rmk}
$E(\mathbb{C}_\infty)$ is equipped with an $\mathbb{F}_q[t]$-module structure via the map $\phi$.
\end{rmk}
For a subring $R$ of $\mathbb{C}_\infty$ containing $A$, we say that the $t$-module $E$ is \textit{defined over $R$} if $\alpha_i$ lies in $\Mat_d(R)$ for all $i\geq 0$.

We take the $n$th tensor power of the Carlitz $\mathbb{F}_q[t]$-module as an example. Fixing a positive integer $n$, the $n$th tensor power of the Carlitz $\mathbb{F}_q[t]$-module denoted by $\mathbf{C}^{\otimes n}$ is an $n$-dimensional $t$-module defined over $A$ together with the $\mathbb{F}_q$-linear ring homomorphism \[[\cdot]_n:\mathbb{F}_q[t]\to \Mat_n(\mathbb{C}_\infty[\tau])\] given by \[[t]_n=\theta I_n+N_n+E_n\tau,\] where\[N_n:=\begin{pmatrix}
 0& 1& \dots& 0\\ 
 \vdots& \ddots& \ddots& \vdots\\
 \vdots& & \ddots& 1\\
 0& \dots& \dots& 0\\
\end{pmatrix},~E_n:=\begin{pmatrix}
 0& \dots& \dots& 0\\ 
 \vdots& & & \vdots\\
 \vdots& & & \vdots\\
 1& \dots& \dots& 0\\
\end{pmatrix}.\]

\subsection{Identification of $\Ext_{\mathscr{F}}^1(\mathbf{1},M')$ and the Anderson $t$-module $E'$}

By an \textit{Anderson $t$-motive} we mean an object $N$ in $\mathscr{F}$ satisfying the following properties.
\begin{enumerate}[(1)]
\item $N$ is a free left $\bar{k}[\sigma]$-module of finite rank.
\item $(t-\theta)^nN\subseteq \sigma N$ for all sufficiently large integers $n$.
\end{enumerate}

\begin{rmk}
We can check directly that $M'$ is an Anderson $t$-motive.
\end{rmk}

Since $M'$ is an Anderson $t$-motive, we can construct an associated Anderson $t$-module $(E',\rho)$ and have the following $\mathbb{F}_q[t]$-module isomorphisms established by Anderson
\[\Ext_{\mathscr{F}}^1(\mathbf{1},M')\cong M'/(\sigma-1)M'\cong E'(\bar{k}).\] We state the details in the Theorem \ref{iso1} and Theorem \ref{iso2} which appeared in \cite{2014arXiv1411.0124C}.

\begin{thm}[{\cite[Thm. 5.2.1]{2014arXiv1411.0124C}}]\label{iso1}
Let $\{ x_0,\dots,x_{n-1}\}$ be a $\bar{k}[t]$-basis of $M'$ on which the $\sigma$-action is presented by the matrix $\Phi'$. Let $M\in\Ext_{\mathscr{F}}^1(\mathbf{1},M')$ be defined by the matrix \[\begin{pmatrix}
 \Phi'& 0\\ 
 f_0,\dots,f_{n-1}& 1\\
\end{pmatrix}.\] Then the map \[\mu:\Ext_{\mathscr{F}}^1(\mathbf{1},M')\to M'/(\sigma-1)M'\] defined by \[\mu(M):=f_0x_0+\cdots+f_{n-1}x_{n-1}\] is an isomorphism of $\mathbb{F}_q[t]$-modules.
\end{thm}

We consider the $n$th tensor power of the Carlitz motive $\mathcal{C}^{\otimes n}\in\mathscr{F}$. The underlying $\bar{k}[t]$-module of $\mathcal{C}^{\otimes n}$ is $\bar{k}[t]$ subject to the $\sigma$-action \[\sigma(f):=(t-\theta)^n f^{(-1)},~f\in\mathcal{C}^{\otimes n}.\] Note that $M'$ fits into the short exact sequence of Frobenious modules \[\begin{tikzcd}
0\arrow[r]
&\mathcal{C}^{\otimes n}\arrow{r}
&M'\arrow[r]
&\bigoplus _{i=1}^{n-1}\mathcal{C}^{\otimes (n-i)}\arrow[r]
&0
\end{tikzcd}
,\] where the projection map is defined by $\sum_{i=0}^{n-1}f_ix_i\mapsto (f_1,\dots,f_{n-1}).$ As a left $\bar{k}[\sigma]$-module, $\mathcal{C}^{\otimes j}$ is free of rank $j$ with the natural basis $\{(t-\theta)^{j-1},\dots,(t-\theta),1\}$. Hence the set \[\left\{ (t-\theta)^{n-1}x_0,\dots,(t-\theta)x_0,x_0,\dots,(t-\theta)x_{n-2},x_{n-2},x_{n-1}\right\},\] denoted by $\{\nu_1,\dots,\nu_d\}$, is a $\bar{k}[\sigma]$-basis of $M'$.

Define the homomorphism of $\mathbb{F}_q$-vector spaces $\Delta :M'\to\Mat_{d\times 1}(\bar{k})$ by \[m=\sum_{i=1}^du_i\nu_i\mapsto\Delta (m):=\begin{pmatrix}
 \delta (u_1)\\
\vdots\\
\delta (u_d)
\end{pmatrix},\] where \[\delta \left(\sum_i\sigma^ic_i^{q^i}\right)=\sum_ic_i^{q^i}.\] We note that the homomorphism $\Delta$ is surjective since \[\Delta (a_1\nu_1+\cdots+a_d\nu_d)=\begin{pmatrix}
a_1\\ 
\vdots\\
a_d
\end{pmatrix}\] for $(a_1,\dots,a_d)^{tr}\in \Mat_{d\times 1}(\bar{k})$. As $t(\sigma-1)M'\subseteq (\sigma-1)M'$, the map $\Delta$ induces an $\mathbb{F}_q[t]$-module structure on $\Mat_{d\times 1}(\bar{k}).$ We denote by $(E',\rho)$ the $t$-module defined over $\bar{k}$ with $E'(\bar{k})$ identified with $\Mat_{d\times 1}(\bar{k})$ on which the $\mathbb{F}_q[t]$-module structure is given by \[\rho:\mathbb{F}_q[t]\to\Mat_d(\bar{k}[\tau])\] so that \[\Delta\left(t(a_1\nu_1+\cdots+a_d\nu_d)\right)=\rho_t\begin{pmatrix}
a_1\\ 
\vdots\\
a_d
\end{pmatrix}.\]

\begin{thm}[{\cite[Thm. 5.2.3]{2014arXiv1411.0124C}}]\label{iso2}
Let $M'$ be the Frobenius module defined by the matrix $\Phi'$. Let $(E',\rho)$ be the $t$-module whose $\bar{k}$-valued points are $E'(\bar{k})$ identified with $\Mat_{d\times 1}(\bar{k})$, which is equipped with the $\mathbb{F}_q[t]$-module structure via $\rho : \mathbb{F}_q[t]\to\Mat_d(\bar{k}[t])$ through the map $\Delta$ as above. Then we have the following isomorphism of $\mathbb{F}_q[t]$-modules \[M'/(\sigma-1)M'\cong E'(\bar{k}).\] 
\end{thm}

For example, we consider the $n$th tensor power of Carlitz motive $\mathcal{C}^{\otimes n}$. As a left $\bar{k}[\sigma]$-module, $\mathcal{C}^{\otimes n}$ is free of rank $n$ with basis $\{(t-\theta)^{n-1},\dots,(t-\theta),1\}$. We let \[\Delta_n:\mathcal{C}^{\otimes n}\to \Mat_{n\times 1}(\bar{k})\] be defined as above with respect to this basis. For $(a_1,\dots,a_n)^{\tr}\in\Mat_{n\times 1}(\bar{k})$, we let \[f=a_1(t-\theta)^{n-1}+\cdots+a_{n-1}(t-\theta)+a_n,\] so that $\Delta_n(f)=(a_1,\dots,a_n)^{\tr}$. We can check directly that the multiplication by $t$ on $\Mat_{n\times 1}(\bar{k})$ is given by \[t\cdot
\begin{pmatrix}
a_1\\ 
\vdots\\
a_n
\end{pmatrix}=\Delta_n(tf)=[t]_n
\begin{pmatrix}
a_1\\ 
\vdots\\
a_n
\end{pmatrix}.\] Hence we have the identification \[\mathcal{C}^{\otimes n}/(\sigma-1)\mathcal{C}^{\otimes n}\simeq\mathbf{C}^{\otimes n}(\bar{k}).\]

\subsection{A criterion for the $\mathbb{F}_q[t]$-torsion property of $M_\mathfrak{C}$ in $\Ext_{\mathscr{F}}^1(\mathbf{1},M')$}

In this section, we provide a criterion for the $\mathbb{F}_q[t]$-torsion property of $M_\mathfrak{C}$ in $\Ext_{\mathscr{F}}^1(\mathbf{1},M')$. The strategy of proof is to follow \cite[Theorem 6.1.1]{2014arXiv1411.0124C}, and so we put most of the proofs\ in the Appendix \ref{pf56}.

\begin{thm}\label{criterion}
Let $(E',\rho)$ be the associated $t$-module given above. Given $\mathfrak{C}\in k^n$, and let $\mathbf{v}_\mathfrak{C}$ be the integral point in $E'(A)$ corresponding to $M_\mathfrak{C}$ via isomorphisms described in the Theorem \ref{iso1} and Theorem \ref{iso2}. For $(q-1)\mid i$, we decompose \[i=p^{\ell_i}n_i(q^{h_i}-1)\] such that $p\nmid n_i$ and $h_i$ is the greatest integer for which $(q^{h_i}-1)\mid i$. Put\[a=\prod(t^{q^{h_i}}-t)^{p^{\ell_i}},\] where the product is taken over integers $i$ from $1$ to $n$ which are multiples of $(q-1)$. Then $M_\mathfrak{C}$ is an $\mathbb{F}_q[t]$-torsion class in $\Ext_{\mathscr{F}}^1(\mathbf{1},M')$ if and only if $\rho_a(\mathbf{v}_\mathfrak{C})=0.$
\end{thm}

\begin{proof} It is clear that $\rho_{a}(\mathbf{v}_\mathfrak{C})=0$ implies that $M_\mathfrak{C}$ is an $\mathbb{F}_q[t]$-torsion class in $\Ext_{\mathscr{F}}^1(\mathbf{1},M').$ Now we suppose that $M_\mathfrak{C}$ is an $\mathbb{F}_q[t]$-torsion class in $\Ext_{\mathscr{F}}^1(\mathbf{1},M')$.  First, we follow the method in \cite[P. 307]{2015arXiv151006519C} to derive the short exact sequence of $\mathbb{F}_q[t]$-modules
\[0\to \mathcal{C}^{\otimes n}/(\sigma-1)\mathcal{C}^{\otimes n}\to M'/(\sigma-1)M' \to \bigoplus _{i=1}^{n-1}\mathcal{C}^{\otimes (n-i)}/(\sigma-1)\mathcal{C}^{\otimes (n-i)} \to 0.\] Note that $M'$ fits into the short exact sequence of Frobenious modules \[\begin{tikzcd}
0\arrow[r]
&\mathcal{C}^{\otimes n}\arrow{r}
&M'\arrow[r]
&\bigoplus _{i=1}^{n-1}\mathcal{C}^{\otimes (n-i)}\arrow[r]
&0
\end{tikzcd}
,\] where the projection map is defined by $\sum_{i=0}^{n-1}f_ix_i\mapsto (f_1,\dots,f_{n-1}).$ Note also that the $\mathbb{F}_q[t]$-linear map $\sigma -1$ from $\bigoplus _{i=1}^{n-1}\mathcal{C}^{\otimes (n-i)}$ to itself is injective. By the Snake Lemma, we have our desired short exact sequence of $\mathbb{F}_q[t]$-modules.

By the Theorem \ref{iso2} and previous isomorphisms of $\mathbb{F}_q[t]$-modules, we have \[0\to\mathbf{C}^{\otimes n}(\bar{k})\to E'(\bar{k})\to\bigoplus _{i=1}^{n-1}\mathbf{C}^{\otimes i}(\bar{k})\to 0.\] Let $\pi$ denote the surjective map. By the Theorem \ref{integrl}, $\mathbf{v}_\mathfrak{C}$ is an integral point, and then so is $\pi(\mathbf{v}_\mathfrak{C}).$ In fact, we have $\pi(\mathbf{v}_\mathfrak{C})\in \bigoplus _{i=1}^{n-1}\mathbf{C}^{\otimes i}(k)_{\tor}$ since $M_\mathfrak{C}$ is $\mathbb{F}_q[t]$-torsion by our assumption. By \cite[Prop. 1.11.2]{AT90} and \cite[Lem. 5.1.3]{2014arXiv1411.0124C}, the polynomial \[b:=\prod_{\substack{i\in\{1,\dots,n\}\\ (q-1)\mid i}}(t^{q^{h_i}}-t)^{p^{\ell_i}}\in\mathbb{F}_q[t]\] annihilates $\pi(\mathbf{v}_\mathfrak{C})$. Hence $\rho_b(\mathbf{v}_\mathfrak{C})\in \ker \pi\cong\mathbf{C}^{\otimes n}(\bar{k}).$ By the Theorem \ref{integrl} again, $E'$ is defined over $A$. $\rho_b(\mathbf{v}_\mathfrak{C})$ is also an integral point. Hence $\rho_b(\mathbf{v}_\mathfrak{C})\in \mathbf{C}^{\otimes n}(k)_{\tor}$. By \cite[Prop. 1.11.2]{AT90} and \cite[Lem. 5.1.3]{2014arXiv1411.0124C} again, $\rho_b(\mathbf{v}_\mathfrak{C})$ is annihilated by \[(t^{q^{h_n}}-t)^{p^{\ell_n}}\in\mathbb{F}_q[t]\] if $(q-1)\mid n$, otherwise $\rho_b(\mathbf{v}_\mathfrak{C})=0.$ Therefore $\rho_a(\mathbf{v}_\mathfrak{C})=0$.\end{proof}

\section{Algorithm and computational results}\label{Algorithm}
\subsection{Algorithm}
In this section we provide an algorithm to determine, for the given $n$-tuple of coefficients $\mathfrak{C}=(b_0,a_1,\dots,a_{n-1})\in k^n$, whether $M_\mathfrak{C}$ is $\mathbb{F}_q[t]$-torsion in $\Ext_{\mathscr{F}}^1(\mathbf{1},M')$ or not. (cf. \cite[\S 6]{2014arXiv1411.0124C})

\begin{enumerate}[STEP 1.]
\item[INPUT:] $r,s\in\mathbb{N}$, $n:=r+s$, $p:$ a prime, $q$: a power of $p$, $\mathfrak{C}=(b_0,a_1,\dots,a_{n-1})\in k^n.$

\item Compute the Anderson-Thakur polynomials $H_0,\dots,H_{n-1}$ and the polynomial $a$ as in the Theorem \ref{criterion}.

\item Let $M'$ be the Frobenius module defined by $\Phi'$ as in \eqref{phi} with $\bar{k}[t]$-basis $\{x_0,\dots,x_{n-1}\}$. Let $\{\nu_1,\dots,\nu_d\}$ be the $\bar{k}[\sigma]$-basis of $M'$ given by \[\left\{ (t-\theta)^{n-1}x_0,\dots,(t-\theta)x_0,x_0,\dots,(t-\theta)x_{n-2},x_{n-2},x_{n-1}\right\}.\] Identify $M'/(\sigma-1)M'$ with $\Mat_{d\times 1}(\bar{k})$ via $\{\nu_1,\dots,\nu_d\}.$

\item  Compute $\beta_0$, $\alpha_i$ for $i=1,\dots,n-1$ as in Definition \ref{ab}. Consider \[\left[ \beta_0H^{(-1)}_{n-1}(t-\theta)^n-\gamma_0(H^{(-1)}_{r-1}H^{(-1)}_{s-1}(t-\theta)^{n})\right]x_0+\sum_{i=1}^{n-1}\left(\alpha_iH^{(-1)}_{(n-i)-1}(t-\theta)^{n-i}\right)x_i\]in $M'/(\sigma-1)M'$ and multiply it by the polynomial $a$. Write it as the form\[\sum_{i=1}^du_i\nu_i\] which corresponds to the integral point $\rho_a(\mathbf{v}_\mathfrak{C})=(\delta(u_1),\dots,\delta(u_d))^{tr}\in E'(A)$ via the $\Delta$ map described in the Theorem \ref{iso2}.

\item[OUTPUT:] If $\rho_a(\mathbf{v}_\mathfrak{C})$ is zero, then $M_\mathfrak{C}$ is an $\mathbb{F}_q[t]$-torsion class in $\Ext_{\mathscr{F}}^1(\mathbf{1},M')$; otherwise, $M_\mathfrak{C}$ is not an $\mathbb{F}_q[t]$-torsion class in $\Ext_{\mathscr{F}}^1(\mathbf{1},M').$
\end{enumerate}

\subsection{Examples}
We use Maple to write a program based on the algorithm. By a ``doable'' $n$-tuple $\mathfrak{C}=(b_0,a_1,\dots,a_{n-1})\in k^n$, we mean the computation for determining whether $M_\mathfrak{C}$ is $\mathbb{F}_q[t]$-torsion in $\Ext_{\mathscr{F}}^1(\mathbf{1},M')$ or not can be done within about $10$ minutes.

We recheck Chen's formula \eqref{chen} and search for $\mathfrak{C}\in k^n-\mathbb{F}_p^n$ with the \ref{shu}-property by the Maple program. 
\begin{enumerate} 
    \item For $p=q<30$, $r$, $s<200$, the computation shows that $M_\mathfrak{C}$ is an $\mathbb{F}_q[t]$-torsion class in $\Ext_{\mathscr{F}}^1(\mathbf{1},M')$ where $\mathfrak{C}\in \mathbb{F}_p^n$ is a ``doable'' $n$-tuple of coefficients coming from the Chen's formula \eqref{chen}. This matches the result of Theorem\ref{main}\eqref{main1}.
    
    \item Let $p=q=3$ be given. For $r$, $s<4$ such that $(q-1)\nmid n$, ``doable'' $n$-tuples $\mathfrak{C}=(b_0,a_1,\dots,a_{n-1})\in k^n-\mathbb{F}_p^n$ with degrees of numerators and denominators of $b_0$, $a_i$ less than $4$, we list some data below for $r$, $s$ and the $n$-tuples $\mathfrak{C}$ when $M_\mathfrak{C}$ is $\mathbb{F}_q[t]$-torsion in $\Ext_{\mathscr{F}}^1(\mathbf{1},M')$:

\begin{align*}
(r, s)=(1, 2), &~\mathfrak{C}=(2,\theta^3+2\theta,0)\in A^3-\mathbb{F}_3^3,\\
(r, s)=(1, 2), &~\mathfrak{C}=(0,2\theta^3+\theta,0)\in A^3-\mathbb{F}_3^3,\\
(r, s) = (2, 3), &~\mathfrak{C}=(0,2\theta^3+\theta,0,2,0)\in A^5-\mathbb{F}_3^5,\\
(r, s)=(1, 2), &~\mathfrak{C}=(\frac{\theta^3+2\theta+2}{\theta^3+2\theta},2,0)\in k^3-A^3,\\
(r, s)=(1, 2), &~\mathfrak{C}=(\frac{2\theta^3+\theta+2}{2\theta^3+\theta},1,0)\in k^3-A^3.\\
\end{align*}By Theorem \ref{main}\eqref{main22}, those $n$-tuples $\mathfrak{C}$ have the \ref{shu}-properties, \textit{i.e.}, we have the following shufle relations:

\begin{align*}
\zeta_A(1)\zeta_A(2)-\zeta_A(1,2)-\zeta_A(2,1)={}&2\zeta_A(3)+(\theta^3+2\theta)\zeta_A(1,2),\\
\zeta_A(1)\zeta_A(2)-\zeta_A(1,2)-\zeta_A(2,1)={}&(2\theta^3+\theta)\zeta_A(1,2),\\
\zeta_A(2)\zeta_A(3)-\zeta_A(2,3)-\zeta_A(3,2)={}&(2\theta^3+\theta)\zeta_A(1,4)+2\zeta_A(3,2),\\
\zeta_A(1)\zeta_A(2)-\zeta_A(1,2)-\zeta_A(2,1)={}&\frac{\theta^3+2\theta+2}{\theta^3+2\theta}\zeta_A(3)+2\zeta_A(1,2),\\
\zeta_A(1)\zeta_A(2)-\zeta_A(1,2)-\zeta_A(2,1)={}&\frac{2\theta^3+\theta+2}{2\theta^3+\theta}\zeta_A(3)+\zeta_A(1,2).\\
\end{align*}
\end{enumerate}

\begin{rmk}
Fix $r$, $s\in \mathbb{N}$, $p$ a prime and $q$ a power of the prime $p$. By a naive analogue of the sum shuffle, we mean the $n$-tuples $\mathfrak{C}:=(1,0,\dots,0)\in\mathbb{F}_p^n$ has the \ref{shu}-property. Conjecturally all the $n$-tuples $\mathfrak{C}\in\mathbb{F}_p^n$ having the \ref{shu}-properties come from Chen's formula \eqref{chen}, which implies all the naive analogues of the sum shuffle come from Chen's formula.  From our data, the $n$-tuple $\mathfrak{C}\in\mathbb{F}_p^n$ we found with the \ref{shu}-property comes from Chen's formula \eqref{chen}, so our data do support the conjecture. 
\end{rmk}

\section{Another method}\label{Another}
\begin{defi}
Fixing $n$, an $n$-tuple $\mathfrak{C}=(b_0,a_1,\dots,a_{n-1})\in k^{n}$ is said to have the \textit{\ref{DR}-property} if it fits into the following linear relation among double zeta values and $\zeta_A(n)$:
\begin{equation}\label{DR}\tag{DR}
0=b_0\zeta_A(n)+\sum_{i=1}^{n-1}a_i\zeta_A(i,n-i). 
\end{equation}
\end{defi}

Fixing $r,s\in\mathbb{N}$ and letting $n:=r+s$, by combining a shuffle relation and the relation (\ref{chen}), we have that the $n$-tuple $\mathfrak{C}=(b_0,a_1,\dots,a_{n-1})\in k^{n}$ has the \ref{shu}-property if and only if the $n$-tuple $\widetilde{\mathfrak{C}}=(\widetilde{b}_0,\widetilde{a}_1,\dots,\widetilde{a}_{n-1})\in k^{n}$ has the \ref{DR}-property where \[\widetilde{b}_0:=b_0-1\] and \[\widetilde{a}_i:=\left\{\begin{matrix}
a_i-\left[ (-1)^{s-1}\begin{pmatrix}n-i-1\\s-1\end{pmatrix}+(-1)^{r-1}\begin{pmatrix}n-i-1\\r-1\end{pmatrix}\right] &\text{if } (q-1)\mid(n-i)\\ 
a_i & \text{otherwise}
\end{matrix}\right..\]

Note that finding all possible $n$-tuples in $k^{n}$ with the \ref{DR}-property is equivalent to finding all $n$-tuples in $A^{n}$ having the \ref{DR}-property. The crucial part is to relate the \ref{DR}-property to the $\mathbb{F}_q[t]$-linear relation among some elements in $\mathbf{C}^{\otimes n}(\bar{k})$. More precisely, put
\[\mathscr{V}:=\{(s_1,s_2)\in\mathbb{N}^2: s_1+s_2=n\text{ and } (q-1)\mid s_2\}.\] For convenience, we label $\mathscr{V}$ as \[\mathscr{V}=\{\mathfrak{s}_1,\dots,\mathfrak{s}_{|\mathscr{V}|}\} \] where $|\mathscr{V}|$ is the cardinality of the finite set $\mathscr{V}$. Considering points \[\{\mathbf{v}_n\}\cup\{\Xi_{\mathfrak{s}_i}\}_{i=1}^{|\mathscr{V}|}\subset\mathbf{C}^{\otimes n}(\bar{k})\] described in \cite[Thm. 2.3.1, Thm. 4.1.1]{2015arXiv151006519C}. We separate them into two cases:

\begin{enumerate}
\item[Case I.] $(q-1) \nmid n$:

We consider the $\mathbb{F}_q[t]$-linear relation:\[
[\eta]_n(\mathbf{v_n})+\sum_{i=1}^{|\mathscr{V}|}[\eta_i]_n(\Xi_{\mathfrak{s}_i})=0.\]
By the proof of \cite[Thm. 6.1.1]{2015arXiv151006519C}, we can effectively determine if the tuple of polynomials $(\eta,\eta_1,\dots,\eta_{|\mathscr{V}|}) \in\mathbb{F}_q[t]^{|\mathscr{V}|+1}$ satisfying the above equation. Then by the proof of \cite[Thm. 5.1.1]{2015arXiv151006519C}, we trace back to an $n$-tuple in $A^n$ having the \ref{DR}-property.

\item[Case II.] $(q-1) \mid n$:

In this case, we consider the $\mathbb{F}_q[t]$-linear relation:\[
\sum_{i=1}^{|\mathscr{V}|}[\eta_i]_n(\Xi_{\mathfrak{s}_i})=0.\]
By the proof of \cite[Thm. 6.1.1]{2015arXiv151006519C} again, we can effectively determine if the tuple of polynomials $(\eta_1,\dots,\eta_{|\mathscr{V}|}) \in\mathbb{F}_q[t]^{|\mathscr{V}|}$ satisfying the above equation. We want to use the proof of \cite[Thm. 5.1.1]{2015arXiv151006519C} to trace back to an $n$-tuple in $A^n$ with the \ref{DR}-property. Unfortunately, we can not determine the first coordinate of the $n$-tuple in $A^n$ with the \ref{DR}-property derived by this process although we know the other coordinates of the $n$-tuple in $A^n$.
\end{enumerate}

In conclusion, we can achieve the same result by the arguments provided in \cite{2015arXiv151006519C}. \textit{i.e.}, we also have an effective criterion to determine whether an $n$-tuple in $k^n$ has the \ref{shu}-property if $(q-1)\nmid n$. For the case $(q-1)\mid n$, we can not explicit determine the first coordinate of an $n$-tuple in $k^n$.

\appendix
\section{The crucial Theorem in proving the Theorem \ref{criterion}}\label{pf56}
\subsection{Two crucial properties}
To derive the Theorem \ref{criterion}, we follow the strategy in \cite[\S 5.3]{2014arXiv1411.0124C}. We need two important properties stated them below.

Via the isomorphisms \[\Ext_{\mathscr{F}}^1(\mathbf{1},M')\cong M'/(\sigma-1)M'\cong E'(\bar{k}),\] we denote the image of the class $M_\mathfrak{C}\in\Ext_{\mathscr{F}}^1(\mathbf{1},M')$ in $E'(\bar{k})$ by $\mathbf{v}_\mathfrak{C}$.

\begin{thm}\label{integrl}
We have that
\begin{enumerate}[(1)]
\item The associated $t$-module $E'$ given above is defined over $A$.
\item $\mathbf{v}_\mathfrak{C}$ is an integral point in $E'(A).$
\end{enumerate}
\end{thm}

In \cite[Thm. 5.3.2]{2014arXiv1411.0124C}, they constructed a special set $\Xi\in M'$ such that the image of $\Xi$ via $\Delta$ is obviously in $E'(A).$ Furthermore, they proved that the special point is an image of some element in $\Xi$ and then completed the proof. Here, we follow the same approach.

\begin{prop}\label{lprop}
Let $M'$ be the Frobenius module defined by the matrix $\Phi'$ in (\ref{phi}) with a $\bar{k}[t]$-basis $x_0,\dots,x_{n-1}.$ Let $\{\nu_1,\dots,\nu_d\}$ be the $\bar{k}[\sigma]$-basis of $M'$ given by \[\left\{ (t-\theta)^{n-1}x_0,\dots,(t-\theta)x_0,x_0,\dots,(t-\theta)x_{n-2},x_{n-2},x_{n-1}\right\}.\]
Let $\Xi$ be the set consisting of all elements in $M'$of the form $\sum_{i=1}^de_i\nu_i$, where $e_j=\sum_n\sigma^nu_{nj}$ with each $u_{nj}\in A$. Then for any nonzero $f\in A[t]$ and any $1\leq \ell\leq n-1$, we have $fx_\ell\in \Xi.$
\end{prop}

\begin{proof}(cf. \cite[Thm. 5.3.2]{2014arXiv1411.0124C})
We first prove the case when $\ell=0$. We divide $f$ by $(t-\theta)^n$ and write \[f=g_1(t-\theta)^n+\gamma_1,\] where $g_1,\gamma_1\in A[t]$ with $\deg_t\gamma_1<n.$ So \[fx_0=g_1\sigma x_0+\gamma x_0=\sigma g_1^{(1)}x_0+\gamma_1 x_0.\] Note that by expanding $\gamma_1$ in terms of powers of $(t-\theta)$ we see that $\gamma_1x_0$ is an $A$-linear combination of $\{\nu_1,\dots,\nu_n\}$.

Next we divide $g_1^{(1)}\in A[t]$ by $(t-\theta)^n$ and write \[g_1^{(1)}=g_2(t-\theta)^n+\gamma_2,\] where $g_2,\gamma_2\in A[t]$ with $\deg_t\gamma_2<n.$ So \[\sigma g_1^{(1)}x_0=\sigma(g_2(t-\theta)^n+\gamma_2)x_0=\sigma^2g_1^{(1)}x_0+\sigma\gamma_2x_0.\] By expanding $\gamma_2$ in terms of $(t-\theta)$ we see that $\sigma\gamma_2x_0\in\Xi.$ By dividing $g_2^{(1)}$ by $(t-\theta)^n$ and continuing the procedure as above inductively we eventually obtain that $fx_0\in\Xi.$

Now for $\ell\geq 2$ we suppose that multiplication by any element of $A[t]$ on $x_i$ belongs to $\Xi$ for $1\leq i\leq \ell-1$. We prove that $fx_\ell\in\Xi$ by the induction on the degree of $f$ in $t$, and note that the result is valid when $\deg_tf\leq n-1-\ell$ by expanding $f$ in terms of powers of $(t-\theta).$ So we suppose that $\deg_tf\geq n-1-\ell+1.$

We divide $f$ by $(t-\theta)^{n-\ell}$ and write \[f=g_1(t-\theta)^{n-\ell}+\gamma_1,\] where $g_1,\gamma_1\in A[t]$ with $\deg_t\gamma_1<n-\ell.$ It follows that \begin{align*}
fx_\ell={}& g_1(t-\theta)^{n-\ell}x_\ell+\gamma_1x_\ell\\
={}& g_1\left[ \sigma x_\ell-H_{(\ell-1)-1}^{(-1)}(t-\theta)^nx_0\right]+\gamma_1x_\ell\\
={}& \sigma g_1^{(1)}\left( x_\ell-H_{(\ell-1)-1}x_0\right)+\gamma_1x_\ell\\
={}& \sigma g_1^{(1)}x_\ell-\sigma g_1^{(1)}H_{(\ell-1)-1}x_0+\gamma_1x_\ell.
\end{align*}
However, by expanding $\gamma_1$ in terms of powers of $(t-\theta)$ we see that $\gamma_1x_\ell\in\Xi$, and by hypothesis $\sigma g_1^{(1)}H_{(\ell-1)-1}x_0\in A[t]$. Thus to prove the desired result we reduce to prove that $g_1^{(1)}x_\ell\in A[t],$ which is valid by the induction hypothesis since
$\deg_tg_1^{(1)}=\deg_1g<\deg_tf.$  
\end{proof}

\begin{rmk}\label{in}
By the definition of $\Delta$ map, $\Delta(\Xi)\subseteq E'(A).$
\end{rmk}

Now, we can prove the Theorem \ref{integrl}.

\begin{proof}[proof of the Theorem \ref{integrl}](cf. \cite[Thm. 5.3.4]{2014arXiv1411.0124C})
\begin{enumerate}[(1)]
\item Given any point $(a_1,\dots,a_d)^{\tr}\in E'(\bar{k})$, its corresponding element in $M'/(\sigma-1)M'$ has a representative of the from $a_1\nu_1+\cdots+a_d\nu_d.$ We claim that the element \[t\left(\sum_{i=1}^da_i\nu_i\right)\] can be expressed as $\sum_{i=1}^db_i\nu_i\in\Xi$ for which each $b_i$ is of the form $b_i=\sum_j\sigma^jc_j$ so that $c_j$ is an $A$-linear combination of $q^{(\cdot)}$th powers of the $a_n$'s. Then via the map $\Delta$, the claim implies that the $t$-module $E'$ is defined over $A$.

We observe that if some \[\nu_i\notin\mathscr{S}:=\left\{(t-\theta)^{n-1}x_0,\dots,(t-\theta)x_{n-2},x_{n-1}\right\},\] then \[ta_i\nu_i=a_i(t-\theta)\nu_i+\theta a_i\nu_i=a_i\nu_{i-1}+\theta a_i\nu_i.\]
Therefore we reduce the claim to the case $\nu_i\in \mathscr{S}.$ To simplify the notation, we denote \[\nu_{i_1}:=x_{n-1},\dots,\nu_{i_n}:=(t-\theta)^{n-1}x_0.\] Now given any $1\leq \ell \leq n$ we consider
$ta_{i_\ell}\nu_{i_\ell}=a_{i_\ell}t(t-\theta)^{\ell-1}x_{n-\ell}$. Applying Proposition \ref{lprop} to $t(t-\theta)^{\ell-1}x_{n-\ell}$ we see that $ta_{i_\ell}\nu_{i_\ell}$ can be written as the form \[a_{i_\ell}\sum_{j=1}^d\left(\sum_{e_j}\sigma^{e_j}b_{e_j}\right)\nu_j=\sum_{j=1}^d\left(\sum_{e_j}\sigma^{e_j}a_{i_\ell}^{q^{e_j}}b_{e_j}\right)\nu_j\] for some $b_{e_j}\in A$, whence the desired result follows.
\item Note that 
\begin{align*}
\left[ \beta_0H^{(-1)}_{n-1}(t-\theta)^n-\gamma_0H^{(-1)}_{r-1}H^{(-1)}_{s-1}(t-\theta)^{n}\right]x_0+\sum_{i=1}^{n-1}\left(\alpha_iH^{(-1)}_{(n-i)-1}(t-\theta)^{n-i}\right)x_i\\
=\sigma\left(\beta_0H_{n-1}-\gamma_0H_{r-1}H_{s-1}\right)x_0+\sum_{i=1}^{n-1}\left(\sigma\alpha_iH_{(n-i)-1}x_i-\sigma\alpha_iH_{(n-i)-1}H_{i-1}x_0\right).
\end{align*}
Applying Proposition \ref{lprop} to the right hand side of the equation above we see that \[\left[ \beta_0H^{(-1)}_{n-1}(t-\theta)^n-\gamma_0H^{(-1)}_{r-1}H^{(-1)}_{s-1}(t-\theta)^{n}\right]x_0+\sum_{i=1}^{n-1}\left(\alpha_iH^{(-1)}_{(n-i)-1}(t-\theta)^{n-i}\right)x_i\in \Xi.\] Since $\mathbf{v}_{\mathfrak{C}}$ is its image via $\Delta$, the result follows from Remark \ref{in}.
\end{enumerate}
\end{proof}

\bibliographystyle{alpha}
\bibliography{reference}

\begin{thebibliography}{{And}86}

\bibitem[ABP04]{ABP04}
G.~W. {Anderson}, W.~D. {Brownawell}, and M.~A. {Papanikolas}.
\newblock Determination of the algebraic relations among special
  {$\Gamma$}-values in positive characteristic.
\newblock {\em Ann. of Math. (2)}, 160(1):237--313, 2004.

\bibitem[{And}86]{A86}
G.~W. {Anderson}.
\newblock $t$-motives.
\newblock {\em Duke Math. J.}, 53(2):457--502, 1986.

\bibitem[AT90]{AT90}
G.~W. {Anderson} and D.~S. {Thakur}.
\newblock Tensor powers of the {C}arlitz module and zeta values.
\newblock {\em Ann. of Math. (2)}, 132(1):159--191, 1990.

\bibitem[AT09]{AT09}
G.~W. {Anderson} and D.~S. {Thakur}.
\newblock Multizeta values for {$\mathbb{ F}_q[t]$}, their period
  interpretation, and relations between them.
\newblock {\em Int. Math. Res. Not. IMRN}, 2009(11):2038--2055, 2009.

\bibitem[{Car}35]{Car35}
L.~{Carlitz}.
\newblock On certain functions connected with polynomials in a {G}alois field.
\newblock {\em Duke Math. J.}, 1(2):137--168, 1935.

\bibitem[{Cha}14]{Chang14}
C.-Y. {Chang}.
\newblock Linear independence of monomials of multizeta values in positive
  characteristic.
\newblock {\em Compos. Math.}, 150(11):1789--1808, 2014.

\bibitem[{Cha}16]{2015arXiv151006519C}
C.-Y. {Chang}.
\newblock {Linear relations among double zeta values in positive
  characteristic}.
\newblock {\em Camb. J. Math.}, 4(3):289--331, 2016.

\bibitem[{Che}15]{Chen2015153}
H.-J. {Chen}.
\newblock On shuffle of double zeta values over {$\mathbb{F}_q[t]$}.
\newblock {\em J. Number Theory}, 148:153--163, 2015.

\bibitem[CPY19]{2014arXiv1411.0124C}
C.-Y. {Chang}, M.~A. {Papanikolas}, and J.~{Yu}.
\newblock An effective criterion for {E}ulerian multizeta values in positive
  characteristic.
\newblock {\em J. Eur. Math. Soc. (JEMS)}, 21(2):405--440, 2019.

\bibitem[{Gos}96]{Goss96}
D.~{Goss}.
\newblock {\em Basic structures of function field arithmetic}, volume~35.
\newblock Springer-Verlag, Berlin, 1996.

\bibitem[{Pap}08]{P08}
M.~A. {Papanikolas}.
\newblock Tannakian duality for {A}nderson-{D}rinfeld motives and algebraic
  independence of {C}arlitz logarithms.
\newblock {\em Invent. Math.}, 171(1):123--174, 2008.

\bibitem[{Tha}04]{Thakurfunctionfield}
D.~S. {Thakur}.
\newblock {\em Function field arithmetic}.
\newblock World Scientific Publishing Co., Inc., River Edge, NJ, 2004.

\bibitem[{Tha}09]{Thakur09}
D.~S. {Thakur}.
\newblock Power sums with applications to multizeta and zeta zero distribution
  for {$\mathbb {F}_q[t]$}.
\newblock {\em Finite Fields Appl.}, 15(4):534--552, 2009.

\bibitem[{Tha}10]{Thakurshuffle}
D.~S. {Thakur}.
\newblock Shuffle relations for function field multizeta values.
\newblock {\em Int. Math. Res. Not. IMRN}, 2010(11):1973--1980, 2010.

\bibitem[Zha16]{Z_MZV}
J.~Zhao.
\newblock {\em Multiple zeta functions, multiple polylogarithms and their
  special values}, volume~12 of {\em Series on Number Theory and its
  Applications}.
\newblock World Scientific Publishing Co. Pte. Ltd., Hackensack, NJ, 2016.

\end{thebibliography}

\end{document}